\documentclass[12pt]{article}
\usepackage{amsmath, amssymb,amsthm}%,amsbsy
\usepackage{graphicx}
\usepackage{amsthm}

\usepackage{enumerate}

\newcommand{\ex}{\mathrm{ex}}
\newcommand{\Ex}{\mathrm{Ex}}
\newcommand{\co}{\mathrm{Co}}

\newtheorem{theorem}{Theorem}
\newtheorem{lemma}{Lemma}

\newtheorem{defin}{Definition}
\newtheorem{fact}{Fact}

\newtheorem{cor}{Corollary}
\def\le{\leqslant}
\def\ge{\geqslant}

\begin{document}

\title{Multicolor Ramsey numbers and restricted Tur\'an numbers for the loose 3-uniform path of length three}

\author{Eliza Jackowska
\\A. Mickiewicz University\\Pozna\'n, Poland\\\tt  elijac@amu.edu.pl
\and Joanna Polcyn
\\A. Mickiewicz University\\Pozna\'n, Poland\\{\tt joaska@amu.edu.pl}
 \and Andrzej Ruci\'nski\thanks{Research supported by the Polish NSC
grant 2014/15/B/ST1/01688.
}
\\A. Mickiewicz University\\Pozna\'n, Poland\\{\tt rucinski@amu.edu.pl}
}

\date{\today}

\maketitle

\begin{abstract}
Let $P$ denote  a 3-uniform hypergraph consisting of 7 vertices $a,b,c,d,e,f,g$ and 3 edges
$\{a,b,c\}, \{c,d,e\},$ and $\{e,f,g\}$. It is known that the $r$-colored Ramsey number for $P$ is
$R(P;r)=r+6$ for $r=2,3$, and that $R(P;r)\le 3r$ for all $r\ge3$. The latter result follows by a
standard application of the Tur\'an number $ex_3(n;P)$, which was determined to be $\binom{n-1}2$
in our previous work. We have also shown that the full star is the only extremal 3-graph for $P$.
In this paper, we perform a subtle analysis of the Tur\'an numbers for $P$ under some additional
restrictions. Most importantly, we determine the largest number of edges in an $n$-vertex $P$-free
3-graph which is not a star. These Tur\'an type results, in turn, allow us to confirm the formula
$R(P;r)=r+6$ for $r\in\{4,5,6,7\}$.
\end{abstract}

%--------------------------------------------------------------------------------------------------------------------------------------------------
\section{Introduction}\label{intro}

In this paper we prove results about both Ramsey numbers and Tur\'an numbers for the \emph{loose
3-uniform path of length 3} defined as the hypergraph $P:=P^3_3$ consisting of 7 vertices, say,
$a,b,c,d,e,f,g$, and 3 edges $\{a,b,c\}, \{c,d,e\},$ and $\{e,f,g\}$. This is a very special case
of a more general notion of the $k$-uniform \textit{loose path $P^k_m$} of length $m$, where
$k,m\ge2$, defined as a $k$-uniform hypergraph (or $k$-graph, for short) with $m$ edges which can
be linearly ordered in such a way that every two consecutive edges intersect in exactly one vertex
while all other pairs of edges are disjoint. Note that some authors, e.g., in \cite{FJS,Kostochka}
call such paths \emph{linear}, while by loose they mean paths in which consecutive edges may intersect
on more vertices.

\textit{The complete $k$-graph} \textit{$K^k_n$} is a $k$-graph on $n$ vertices in which every
$k$-element subset of the vertex set forms an edge. For a given $k$-graph $F$ and an integer
$r\ge2$, the \textit{Ramsey number} $R(F;r)$ is the least integer $n$ such that  every $r$-coloring
of the edges of $K^k_n$ results in a monochromatic copy of $F$.

In the classical case of two colors ($r=2$), it is known already that for graphs ($k=2$)
$$R(P^2_m;2)\overset{\cite{OnRamsey}}{=}\left\lfloor\frac{3m+1}{2}\right\rfloor\;,\quad\mbox{ while for 3-graphs }\quad
 R(P^3_m; 2)\overset{\cite{Omidi}}{=}\left\lfloor\frac{5m+1}{2}\right\rfloor,$$
 both formulae holding for all $m\ge2$.
For higher dimensions ($k\ge4$),  only the numbers $R(P^k_m;2)$, $m=2,3,4$, have been determined
exactly (see \cite{GR}), while in \cite{GSS} an asymptotic formula $R(P^k_m;2)\sim (k-1/2)m$, $k$
fixed, $m\to\infty$, was established. For more than two colors, the only existing results are
$R(P;3)=9$ and $r+6\le R(P;r)\le3r$ for $r\ge3$ \cite{J, JPR}. We include below a simple proof of
the upper bound to recall the standard technique of using Tur\'an numbers for bounding Ramsey
numbers, as this is the starting point of the research presented in this paper.

 For a given $k$-graph $F$ and an integer $n\ge1$,  \textit{the Tur\'an number} $ex_k(n;F)$ is the largest number of edges in an $n$-vertex $F$-free $k$-graph
 (for a more general definition, see Section \ref{turan}.) Every $n$-vertex $F$-free $k$-graph with $ex_k(n;F)$ edges is called  \emph{extremal}.

 Clearly, if $\binom nk>r\cdot ex_k(n;F)$, then $R(F;r)\le n$.
 This trivial observation can sometimes be sharpened, owing to a specific structure of the extremal $k$-graphs.
 \emph{A star} is a hypergraph with a vertex, called \emph{the center}, contained in all the edges.
  An $n$-vertex $k$-uniform star is called \emph{full} and denoted by $S_n^k$ if it has $\binom{n-1}{k-1}$ edges.

  It has been proved in \cite{JPR} that for $n\ge8$, $ex_3(n;P)=\binom{n-1}2$ and that $S_n^3$
  is the only extremal $3$-graph.
 Thus, the above inequality is equivalent to  $n>3r$ and yields only that $R(P;r)\le 3r+1$.
If $n=3r$, then $\binom n3=r\cdot ex_3(n;P)$, meaning that for every $r$-coloring of $K_n^3$ either
there is a monochromatic copy of $P$ or every color forms a full star which, however, is
impossible.
 This was good enough to claim that $R(P;3)=9$ in \cite{J}, but  for $r=4$ it only yielded the bound $R(P;4)\le12$. To make further progress in pin-pointing the Ramsey numbers $R(P;r)$ one has to refine the analysis of the Tur\'an numbers  and extremal $3$-graphs for $P$ which, in our opinion, might be of independent interest.

Let us illustrate our approach by sticking to the case $r=4$ for a while.
 The lower bound on $R(P;4)$ is $r+6=10$ and $\tfrac14\binom{10}3=30<\binom92$. This only tells us that in every 4-coloring of $K_{10}^3$ a color must have been applied to at least 30 edges. If we only knew that the edges of that color formed a star (not necessarily full), then we could remove the center of that star reducing the picture to a 3-coloring of $K_9^3$ about which we already know that it does contain a monochromatic copy of $P$.

 In this paper we prove that this is, indeed, the case. In fact, we prove a much stronger result by determining precisely the largest number of edges in an $n$-vertex $P$-free $3$-graph which is not a subset of a star. We call this  \emph{ the Tur\'an number of the second order}. This approach works fine for $r=5$ and $r=7$, but, quite surprisingly, fails for $r=6$. In this case, we need to define \emph{ the Tur\'an number of the third order} and compute it for $n=12$.

 Our contribution to the Ramsey theory of hypergraphs is summarized in the following result.

 \begin{theorem}\label{main}
    For all $r\le7$, $R(P; r)=r+6$.
\end{theorem}

 In the next section we  define Tur\'an numbers of the $s$-th order, $s\ge1$, as well as, conditional Tur\'an numbers, and state several  results about them with respect to the path $P$.
 Then in Section \ref{proofRam}, using some of these results,  we prove Theorem \ref{main}.
 The remaining sections  are all devoted to proving the Tur\'an-type theorems from Section \ref{turan}.

\section{Tur\'an numbers}\label{turan}

In this section, after providing some background, we define Tur\'an numbers of the $s$-th order as
well as conditional Tur\'an numbers, and formulate our results concerning such numbers for $P$, the
loose 3-uniform path of length 3. We begin by recalling the definition of the ordinary Tur\'an
number. Given a family of $k$-graphs $\mathcal F$, we call a $k$-graph $H$ \emph{$\mathcal F$-free}
if for all $F \in \mathcal F$ we have  $F \nsubseteq H$.
\begin{defin}
    \rm For a family of $k$-graphs $\mathcal F$ and an integer $n\ge1$, the \textit{ Tur\'an number (of the 1st order)} is defined as
    $$
    \mathrm{ex}^{(1)}_k(n; \mathcal F):=\mathrm{ex}_k(n;\mathcal F)=\max\{|E(H)|:|V(H)|=n\;\mbox{ and $H$ is
    $\mathcal F$-free}\}.
    $$
 Every $n$-vertex $\mathcal F$-free $k$-graph with $ex_k(n;\mathcal F)$
 edges is called  \emph{extremal (1-extremal) for~$\mathcal F$}.
  We denote by $\mathrm{Ex}_k(n;\mathcal F)=\mathrm{Ex}^{(1)}_k(n;\mathcal F)$
  the family of all $n$-vertex $k$-graphs which are extremal for $\mathcal F$.
\end{defin}
\noindent In the case when $\mathcal F=\{F\}$, we will often write $\mathrm{ex}_k(n;F)$ for
$\mathrm{ex}_k(n;\{F\})$ and $\mathrm{Ex}_k(n;F)$ for $\mathrm{Ex}_k(n;\{F\})$.

 The Tur\'an numbers for graphs have been harder to grasp in the case of bipartite $F$ than  when $\chi(F)\ge3$. For $k$-graphs, $k\ge3$, on the other hand, the $k$-partite case seems to be easier.
Indeed, the numbers $\mathrm{ex}_k(n;F)$ have been already computed for $F$ being a pair of disjoint edges, a loose path and a loose  cycle, while, e.g.,  $\mathrm{ex}_3(n;K_4^3)$ is still not known, even asymptotically. Interestingly, the three $k$-partite cases of $F$ mentioned above exhibit a whole lot of similarity.

A family ${F}$ of sets is called \emph{intersecting} if $e\cap e' \neq\emptyset$
 for all $e,e'\in {F}$. Obviously, a star is intersecting.  Restricting to $n$-vertex $k$-graphs, a celebrated result of Erd\H os, Ko, and Rado  asserts that for $n\ge2k+1$, the full star $S_n^k$ is, indeed, the unique largest intersecting family. Below, we formulate this result in terms of the Tur\'an numbers. Let $M_2^k$ be a $k$-graph consisting of two disjoint edges.
\begin{theorem}[\cite{EKR}]
    For $n\ge 2k$, $\mathrm{ex}_k(n;M^k_2)=\binom{n-1}{k-1}$. Moreover, for $n\ge2k+1$, \newline $\mathrm{Ex}_k(n;M^k_2)=\{S^k_n\}$.
\end{theorem}

\emph{A loose cycle $C_m^k$} is defined in the same way as a loose path $P^k_m$, except that this
time also the first and the last edge share one vertex. When $k=m=3$ it is sometimes called \emph{a
triangle}. For convenience we abbreviate  our notation for triangles to $C:=C_3^3$. The Tur\'an
number $\ex_3(n;C)$ has been determined in \cite{FF} for $n\ge 75$ and later for all $n$ in
\cite{CK}.
\begin{theorem}[\cite{CK}]\label{c3}
    For $n\ge 6$, $\ex_3(n;C)=\binom{n-1}{2}$. Moreover, for $n\ge8$, \newline $\mathrm{Ex}_3(n;C)=\{S^3_n\}$.
\end{theorem}

Finally, we return to loose paths. For large $n$, the Tur\'an number for $P^k_m$ has been
determined for $k\ge4$ in \cite{FJS} and for $m\ge4$ in \cite{Kostochka}. In \cite{FJS} the authors
admitted that their method does not quite work for $k=3$, while  the authors of \cite{Kostochka}
credited \cite{FJS} with that case. In \cite{JPR}  we closed this gap. Given two $k$-graphs $F_1$
and $F_2$, by $F_1\cup F_2$ we denote a vertex-disjoint union of $F_1$ and $F_2$. Also, note that
$K_1^3$ is just an isolated vertex.

\begin{theorem}[\cite{JPR}]\label{ex1}
 $$\ex_3(n;P)=\left\{ \begin{array}{ll}
 \binom n3 & \textrm{ and $\quad Ex_3(n;P)=\{K^3_n\}\qquad\quad$\;\; for $n\le6,$ }\\
20 & \textrm{ and $\quad Ex_3(n;P)=\{K^3_6\cup K^3_1\}\quad$ for $n=7,$ }\\
\binom{n-1}{2} & \textrm{ and $\quad Ex_3(n;P)=\{S^3_n\}\qquad\quad$\;\; for $n\ge 8$.}
\end{array} \right.
$$
\end{theorem}

It was proved in \cite{F} for large $n$ and in \cite{KMW} for all $n$ that for $k\ge4$ the Tur\'an number for $P^k_2$, or the maximum number of edges in a $k$-graph with no singleton intersection, is  $ex_k(n;P^k_2)=\binom{n-2}{k-2}$. In a couple of proofs we will need an easy analog of this result for $k=3$, first observed in \cite{KMW}.

\begin{fact}\label{p2}
    For $n\ge 1$, we have $ex_3(n;P_2^3)\le n$.
\end{fact}
%, and the equality holds when  $n$ is divisible by 4 with the unique extremal $3$-graph being the union of disjoint copies of $K_4^3$

\subsection{A hierarchy of Tur\'an numbers}

Tur\'an numbers of the 1st order are just the ordinary Tur\'an numbers defined above. Here we
introduce a hierarchy of Tur\'an numbers, where in each  generation we consider only $k$-graphs
which are not sub-$k$-graphs of extremal $k$-graphs from all previous generations. The next
definition is iterative.

\begin{defin}\rm  For a family of $k$-graphs $\mathcal F$ and  integers $s,n\ge1$,
 \textit{the Tur\'an number of the $(s+1)$-st order} is defined as
\begin{eqnarray*}
\mathrm{ex}^{(s+1)}_k(n;\mathcal F)=\max\{|E(H)|:|V(H)|=n,\; \mbox{$H$ is
    $\mathcal F$-free, and }\\
 \forall H'\in \mathrm{Ex}^{(1)}_k(n;\mathcal F)\cup...\cup\mathrm{Ex}^{(s)}_k(n;\mathcal F),  H\nsubseteq H'\},
\end{eqnarray*}
if such a $k$-graph $H$ exists.
An $n$-vertex $\mathcal F$-free $k$-graph $H$ is called \textit{(s+1)-extremal for} $\mathcal F$ if $|E(H)| = \mathrm{ex}^{(s+1)}_k(n;\mathcal F)$ and  $\forall H'\in \mathrm{Ex}^{(1)}_k(n;\mathcal F)\cup...\cup\mathrm{Ex}^{(s)}_k(n;\mathcal F),  H\nsubseteq H'$; we denote by $\mathrm{Ex}^{(s+1)}_k(n;\mathcal F)$ the family of $n$-vertex $k$-graphs which are $(s+1)$-extremal for $\mathcal F$.
\end{defin}

A historically first example of a Tur\'an number of the 2nd order is due to Hilton and Milner \cite{HM} who determined the maximum size of \emph{a nontrivial} intersecting family of $k$-sets, that is, one which is not a star. We state it here for $k=3$ only and suppress the family $\mathrm{Ex}^{(2)}_3(n;M_2^3)$ which was also found in \cite{HM}. Set $M:=M_2^3$ for convenience.

\begin{theorem}[\cite{HM}]\label{nti}
    For $n\ge6$, we have $\mathrm{ex}^{(2)}_3(n;M)=3n-8$.
\end{theorem}

In this paper we prove the following two results which we then use to compute some Ramsey numbers
for $P$. First, we completely determine $\ex^{(2)}_3(n;P)$, together with the corresponding
2-extremal 3-graphs. \emph{A comet} $\co(n)$ is a 3-graph with $n$ vertices consisting of a copy of
$K_4^3$ to which a star $S_{n-3}^3$ is attached, the unique common vertex being the center of the
star (see Fig. \ref{FigR1}). This vertex is called \emph{the center} of the comet, while the set of
the remaining three vertices of the 4-clique is called \emph{the head}.

\bigskip
\begin{figure}[!ht]
\centering
\includegraphics [width=7cm]{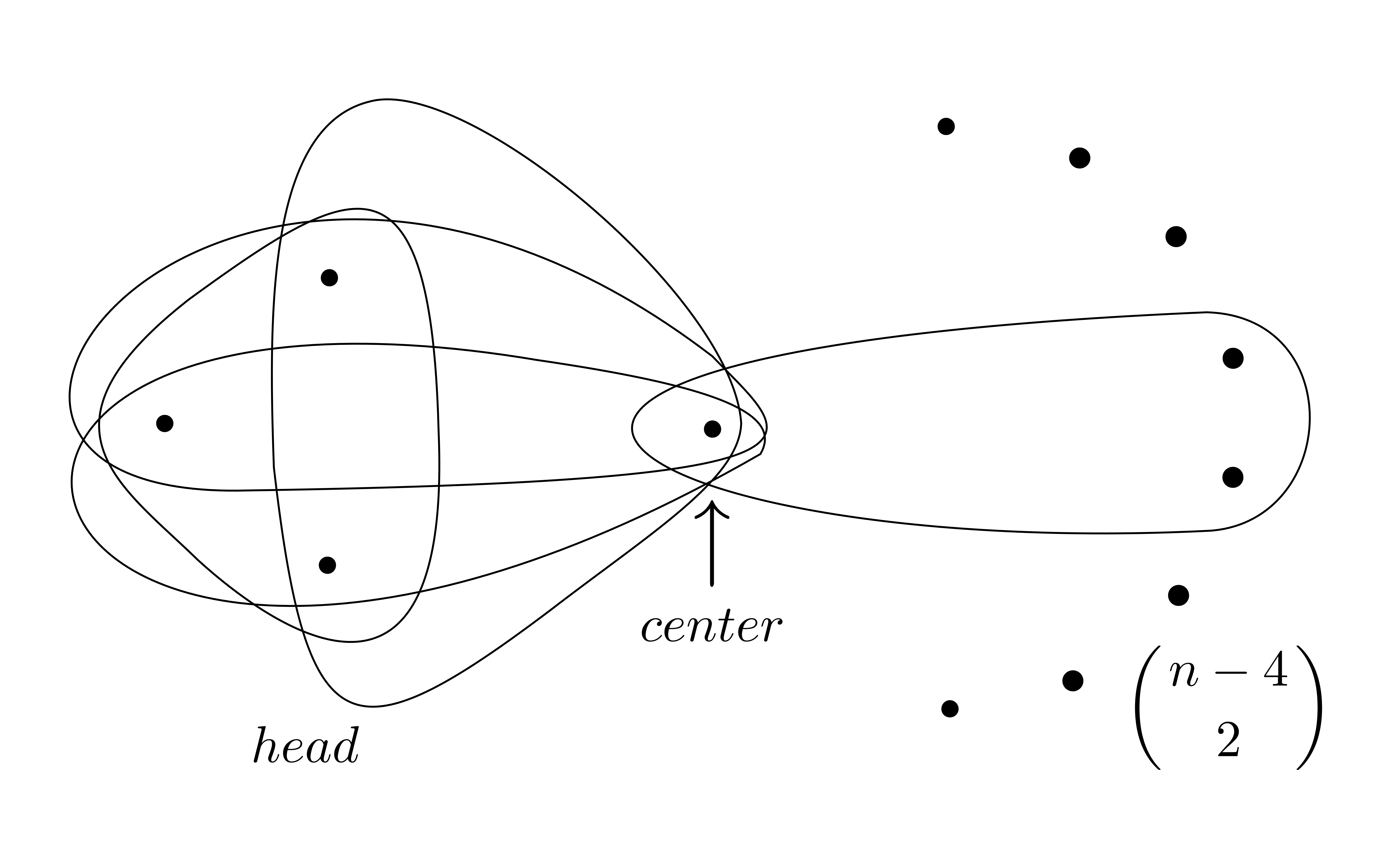}
\caption{The comet $\co(n)$} \label{FigR1}
\end{figure}

\begin{theorem}\label{ex2}
    $$\ex^{(2)}_3(n;P)=\left\{ \begin{array}{ll}
    15 &\textrm{and} \quad  \Ex^{(2)}_3(n;P)=\{S_7^3\}\hskip 4,1cm \textrm{for $n=7$},\\
    20 + \binom {n-6}3 & \textrm{and}  \quad \Ex^{(2)}_3(n;P)=\{K^3_6 \cup K^3_{n-6}\}\quad\qquad\quad \textrm{for $8 \le n\le 12$},\\
    40 & \textrm{and}  \quad \Ex^{(2)}_3(n;P)=\{K^3_6\cup K^3_6\cup K^3_1,\co(13)\} \quad\textrm{for }n=13,\\
    4 + \binom{n-4}{2}&\textrm{and}\quad\Ex^{(2)}_3(n;P)=\{\co(n)\}\hskip 3.4cm\textrm{for }n\ge 14.
    \end{array} \right.
    $$
\end{theorem}
\noindent Note that for $n\le 6$ this number is not defined, since each 3-graph is a sub-3-graph
of~$K_n^3$. Then, we calculate the 3rd Tur\'an number for $P$, but only for $n=12$ which is,
however, just enough for our application.
\begin{theorem}\label{ex3}
    $$\ex^{(3)}_3(12;P)=32 \quad\text{and} \quad \Ex^{(3)}_3(12;P)=\{\co(12)\}.$$
\end{theorem}

\subsection{Conditional Tur\'an numbers}
To determine the Tur\'an numbers of higher order, it is sometimes useful to rely on Theorem \ref{nti} and divide all 3-graphs into those which contain $M$ and those which do not. This leads us quickly to another variation on Tur\'an numbers.
\begin{defin}\rm
     For a family of $k$-graphs $\mathcal F$, a family of $\mathcal F$-free $k$-graphs $\mathcal G$, and an integer $n\ge \min\{|V(G)|:G\in \mathcal G\}$, the  \textit{conditional Tur\'an number} is defined as
     \begin{eqnarray*}
        \ex_k(n;\mathcal F|\mathcal G)=\max\{|E(H)|:|V(H)|=n,\;
      \mbox{$H$ is
    $\mathcal F$-free, and }\exists G\in\mathcal G:\; H\supseteq G \}
     \end{eqnarray*}
     Every $n$-vertex $\mathcal F$-free $k$-graph with $\ex_k(n;\mathcal F|\mathcal G)$ edges and such
     that $H\supseteq G$ for some $G\in\mathcal G$ is called \emph{$\mathcal G$-extremal for $\mathcal F$}.
     We denote by $\Ex_k(n;\mathcal F|\mathcal G)$ the family of all \newline $n$-vertex $k$-graphs which are $\mathcal G$-extremal for $\mathcal F$.
 (If $\mathcal F=\{F\}$ or $\mathcal G=\{G\}$, we will simply write
$\ex_k(n;F| \mathcal G)$, $\ex_k(n;\mathcal F|G)$, $\ex_k(n;F|G)$, $\Ex_k(n;F| \mathcal G)$, $\Ex_k(n;\mathcal F|G)$, or $\Ex_k(n;F|G)$, respectively.)
\end{defin}

In \cite{JPR} we determined  $\ex_3(n;P|C)$ in terms of the ordinary Tur\'an numbers $\ex_3(n;P)$.

\begin{theorem}[\cite{JPR}]\label{second}
For $n\ge6$,
$$\ex_3(n;P|C)=20+\ex_3(n-6;P).$$
Moreover, $Ex_3(n;P|C)=\{K_6^3\cup H_{n-6}\}$, where $Ex_3(n-6;P)=\{H_{n-6}\}$, that is, $H_{n-6}$ is the unique extremal $P$-free 3-graph on
$n-6$ vertices (cf. Theorem \ref{ex1}).
\end{theorem}

Theorem \ref{second}, combined with Theorem \ref{ex1}, yields immediately explicit values of
$ex_3(n;P|C)$ along with the extremal sets $Ex_3(n;P|C)$.

\begin{cor}\label{coro}
$$
\ex_3(n;P|C)=\left\{ \begin{array}{ll}
20 + \binom {n-6}3, & \textrm{and}\quad \Ex_3(n;P|C)=\{K^3_6 \cup K^3_{n-6}\}\quad\textrm{for } 6 \le n\le 12,\\
40, & \textrm{and}\quad\Ex_3(n;P|C)=\{K^3_6 \cup K^3_6\cup K_1^3\} \quad\; \textrm{for }n=13,\\
20 + \binom{n-7}{2}, & \textrm{and}\quad \Ex_3(n;P|C)=\{K^3_6\cup S^3_{n-6}\} \qquad\quad \textrm{for } n\ge 14.
\end{array} \right.
$$
\end{cor}
Our next result reveals that the conditional Tur\'an number $\ex(n;P|C)$ drops significantly if we restrict ourselves  to connected 3-graphs only.
  A 3-graph $H=(V(H),E(H))$ is \emph{connected} if for every bipartition of the set of vertices $V(H)=V_1\cup V_2$, $V_1\neq \emptyset$, $V_2\neq \emptyset$, there exists an edge $h\in E(H)$ such that $h\cap V_1\neq\emptyset$ and $h\cap V_2\neq\emptyset$.

\begin{lemma}\label{spojny}
    If $H$ is a connected $P$-free 3-graph with $n\ge 7$ vertices and $H\supset C$, then
    $$
    |E(H)| \le 3n - 8.
    $$
\end{lemma}
\noindent It is not a coincidence that in Lemma \ref{spojny} and Theorem \ref{nti} we see the same
extremal number $3n-8$. In fact, we prove Lemma \ref{spojny} (see Section \ref{proofspojny}) by
showing that the extremal 3-graph forms a nontrivial intersecting family.

Already in \cite{JPR} we observed that, as a consequence of Theorem \ref{nti},
$$\ex^{(2)}_3(n;P)=\ex_3(n;P|M)\quad\mbox{ and }\quad \ex^{(2)}_3(n;C)=ex_3(n;C|M),$$
except for some very small values of $n$. We also found constructions yielding lower bounds and
conjectured that these bounds are, indeed, the true values (see also Section \ref{final}). In this
paper we confirm one of these conjectures.

\begin{theorem} \label{PM}
    $$
    \mathrm{ex}_3(n;P|M)=\left\{ \begin{array}{ll}
    20 + \binom {n-6}3 & \textrm{and} \quad \Ex_3(n;P|M)=\{K^3_6 \cup K^3_{n-6}\}\quad\quad\quad \textrm{ for }6 \le n\le 12,\\
    40& \textrm{and}\quad \Ex_3(n;P|M)=\{K^3_6\cup K^3_6\cup K^3_1, \co(13)\}\textrm{ for } n=13,\\
    4 + \binom{n-4}{2} &\text{and} \quad \Ex_3(n;P|M)=\{\co(n)\}\qquad\qquad\qquad\quad\;\textrm{for } n\ge 14.
    \end{array} \right.
    $$
\end{theorem}
\noindent Note that the Tur\'an numbers $\mathrm{ex}_3(n;P|M)$ and $\ex^{(2)}_3(n;P)$ coincide for $n\ge8$.

We also find it useful to determine the  Tur\'an number for the pair $\{P,C\}$ conditioning on
3-graphs $H$ being non-intersecting.

\begin{theorem}\label{PCM}
    $$
    \mathrm{ex}_3(n;\{P,C\}|M)=
    \left\{\begin{array}{ll}
    2n-4 &\qquad\qquad\qquad\qquad\qquad\qquad\qquad\;\;\;\;\; \;\,\textrm{for } 6\le n\le 9,\\
    20 &\qquad\qquad\qquad\qquad\qquad\qquad\qquad\qquad\;\;\; \;\,\textrm{for }n=10,\\
    4 + \binom{n-4}{2} &\textrm{and} \quad \Ex_3(n;\{P,C\}|M)=\{\co(n)\}\,\quad\quad\textrm{for } n\ge 11.
    \end{array}\right.
    $$
\end{theorem}
\noindent Note that the Tur\'an numbers $\mathrm{ex}_3(n;\{P,C\}|M)$, $\mathrm{ex}_3(n;P|M)$, and
$\ex^{(2)}_3(n;P)$ coincide for $n\ge13$.

To prove Theorem \ref{PCM} we will need a lemma which states that if one, in addition to $\{P,C\}$,
forbids also $P^3_2\cup K^3_3$, then the formula, valid for $\mathrm{ex}_3(n;\{P,C\}|M)$ only for
$6\le n\le9$, takes over for all values of $n$.

\begin{lemma}\label{pcppm}
    For $n\ge 6$
    $$\ex_3(n;\{P,C,P^3_2\cup K_3^3\}|M)=2n-4.$$
\end{lemma}

\section{Proof of Theorem \ref{main}}\label{proofRam}

As mentioned in the Introduction, the inequality $R(P;r)\ge r+6$, $r\ge1$, has been already  proved
in \cite{J}. We are going to show that $R(P;r)\le r+6$ for each $r=4,5,6,7$.

\bigskip

{\bf Case $r=4$}. Let us consider an arbitrary 4-coloring of the $\binom{10}{3}=120$ edges of the
complete 3-graph $K^3_{10}$.  There exists a color with at least $\frac 14\cdot 120 = 30$ edges.
Denote the set of these edges by $H$. Since, by  Theorem \ref{ex1},
$\Ex_3^{(1)}(10;P)=\{S^3_{10}\}$, and, by Theorem~\ref{ex2}, $\ex^{(2)}(10;P)= 24 < 30$, either
$P\subseteq H$ or $H\subseteq S_{10}^3$. In the latter case we delete the center  of the star
containing $H$, together with the incident edges, obtaining a 3-coloring of $K^3_9$. Since
$R(P;3)=9$, there is a monochromatic copy of $P$.

\bigskip

{\bf Case $r=5$}.
    The proof follows the lines of the previous one. We consider a \newline 5-coloring
 of the complete 3-graph $K^3_{11}$. There exists a color with at least $\binom{11}{3}/5=33$ edges. Denote the set of these edges by $H$. Again, by Theorems \ref{ex1} and  \ref{ex2}, either $P\subseteq H$ or $H\subseteq S^3_{11}$. In the latter case we  delete the center  of the star containing $H$, together with its incident edges, obtaining a 4-coloring of $K^3_{10}$. Since, as we have just proved, $R(P;4)=10$, there is a monochromatic copy of $P$.

\bigskip

{\bf Case $r=6$}.
This is the most difficult case in which we have to appeal to the 3rd Tur\'an number.
 We begin, as before, by considering an arbitrary 6-coloring of
 the complete 3-graph $K^3_{12}$ on the set of vertices $V$ and assuming that it does not yield
  a monochromatic copy of the path $P$. Then none of the color classes can be contained in a star
  $S_{12}^3$, since otherwise we would delete this star, obtaining a 5-coloring of $K_{11}^3$,
  which surely contains a monochromatic $P$.  By Theorems \ref{ex1} and \ref{ex2}, $S_{12}^3$  and
   $K^3_6\cup K^3_6$ are, respectively, the unique 1-extremal and 2-extremal 3-graph for $P$.
    Consequently, by Theorem \ref{ex3}, every color class with more than 32 edges must be a sub-3-graph
      of $K^3_6\cup K^3_6$.

There exists a color class with at least $\left\lceil\binom{12}{6}/6\right\rceil=37$ edges which,
as explained above, is contained in a copy $K$ of $K^3_6\cup K^3_6$. After deleting all the edges
of $K$ from $K^3_{12}$, we obtain a complete bipartite 3-graph $B$ with  bipartition $V=U\cup W$,
$|U|=|W|=6$, and  with $|E(B)|=220-40=180$ edges, colored by 5 colors. Note that any copy of
$K^3_6\cup K^3_6$   may share with $B$ at most 36 edges. Consequently, since $180/5=36$, every
color class has precisely 36 edges and, thus, is contained in $K^3_6\cup K^3_6$.

Let $G_i$, $i=1,2,3,4,5$, be the 5 color classes. Then, for each $i$, $G_i$ is fully characterized
by two partitions, $U=U_i'\cup U_i''$ and $W=W_i'\cup W_i''$. ($G_i$ is then a disjoint union of
two copies of $K_6^3$, one on the vertex set $U_i'\cup W_i'$, the other one on $U_i''\cup W_i''$,
with $U_i',U_i'',W_i', W_i''$ being the 4  missing edges (see Fig. \ref{R2}).)

\bigskip
\begin{figure}[!ht]
\centering
\includegraphics [width=7cm]{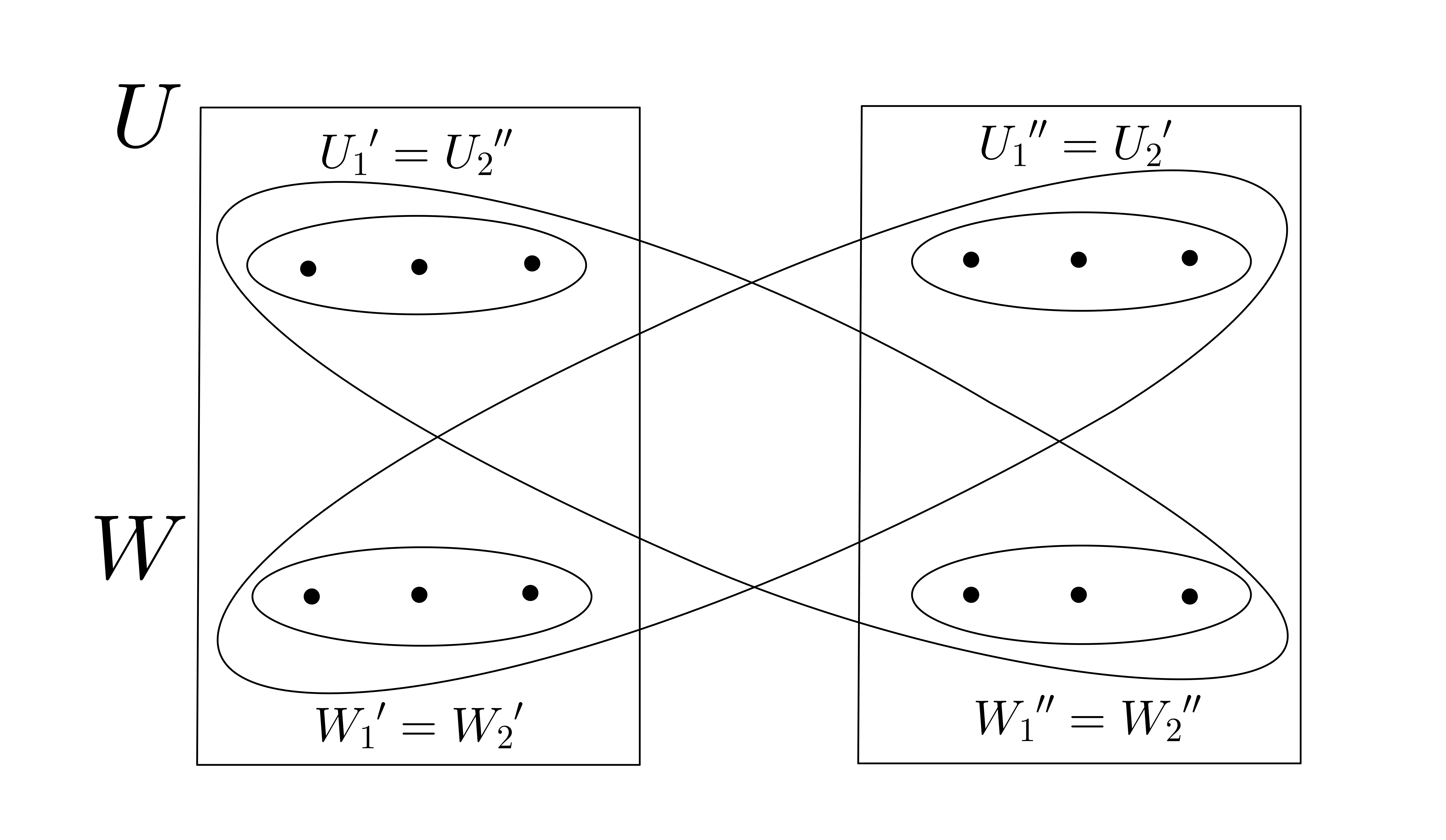}
\caption{Illustration to the proof of Theorem \ref{main}, case $r=6$} \label{R2}
\end{figure}

We now show that only 2 of the 5 color classes can be disjoint which is a contradiction (with a big
cushion). For $G_1$ and $G_2$ to be disjoint, we need that $\{U'_1,U''_1\}=\{U_2',U_2''\}$ and
$\{W'_1,W''_1\}=\{W_2',W_2''\}$, which simply means that one of the partitions, of $U$ or of $W$,
must be swapped. But this implies that  $G_1$, $G_2$, and $G_3$ cannot be pairwise disjoint.

\bigskip

{\bf Case $r=7$.}
As $\left\lceil\binom{13}{3}/7\right\rceil=41>40=\ex^{(2)}(13;P)$, the proof in this case follows the lines of the proofs for $r=4$ and $r=5$, and therefore is omitted.

\section{Proofs of Theorems \ref{ex2}, \ref{ex3}, and \ref{PM}}

In this section we first deduce Theorems \ref{ex2} and \ref{ex3} from Lemma  \ref{spojny} and
Theorems \ref{PM} and~\ref{PCM}, with a little help of some already known results (Theorems
\ref{c3}-\ref{nti}). Then we deduce  Theorem \ref{PM} from Corollary \ref{coro} and Theorem
\ref{PCM}. The proofs of  Lemmas   \ref{spojny}  and \ref{pcppm} will be presented in the next
section, while the proof of the crucial Theorem \ref{PCM}, based on Lemma \ref{pcppm}, is deferred
to the last section.

Throughout all the proofs, for convenience, we will be often identifying the edge set of a 3-graph with the 3-graph itself, writing, e.g., $|H|$ instead of $|E(H)|$.

\bigskip

\noindent{\bf\textit{Proof of Theorem \ref{ex2}.}} We consider the case $n=7$ separately.

\medskip

\noindent $\mathbf{(n=7).}$ By Theorem \ref{ex1},
$$\ex_3^{(1)}(7;P)=20\quad \textrm{and }\quad \mathrm{Ex}_3^{(1)}(7;P)=\{K_6^3\cup K_1^3\}.$$
Therefore,  to determine $\ex_3^{(2)}(7;P)$ we need to find the largest number of edges  in a
$7$-vertex $P$-free 3-graph $H$ which is not a sub-3-graph of $K_6^3\cup K_1^3$. Note that
$P\nsubseteq S_7^3 \nsubseteq K_6^3\cup K_1^3$, and thus,
$$
ex^{(2)}(7;P)\ge |S_7^3|=\binom{7-1}{2}=15.
$$
If $H$ is a 7-vertex $P$-free 3-graph with $|H|>15$, then, by Theorem \ref{c3},
  $H\supset C$. But then, since $H\nsubseteq K_6^3\cup K_1^3$, $H$ must be  connected. Consequently,
 by Lemma \ref{spojny}, $|H|\le 3\times 7-8 = 13$, a contradiction. Checking that  $S_7^3$
  is the unique 2-extremal 3-graph for $P$ and $n=7$ is left to the reader.

\medskip

\noindent $\mathbf{(n\ge8).}$ By Theorem \ref{ex1} we have
$$\ex_3^{(1)}(n;P)=\binom{n-1}{2}\quad \textrm{and}\quad\mathrm{Ex}_3^{(1)}(n;P)=\{S^3_n\}.$$
 Therefore, to determine $\ex_3^{(2)}(n;P)$ for $n\ge 8$ we need
 to find the largest number of edges in an $n$-vertex $P$-free
 3-graph $H$  which is not a subgraph of the star $S^3_n$. If $H$ is an
  intersecting family, then, by Theorem \ref{nti}, $|H|\le \mathrm{ex}^{(2)}_3(n;M) = 3n-8$.
 Otherwise, $H \supset M$ and, therefore, $|H|\le\ex_3(n;P|M)$.
 Using Theorem \ref{PM} one can verify that for $n\ge 8$ we have $\ex_3(n;P|M)>3n-8$.
 Consequently, $$\ex^{(2)}_3(n;P)=\max\{\mathrm{ex}^{(2)}_3(n;M),\;\ex_3(n;P|M)\}=\ex_3(n;P|M)$$
 and Theorem \ref{ex2} for $n\ge8$ follows  by Theorem \ref{PM}. \qed

\bigskip

\noindent{\bf\textit{Proof of Theorem \ref{ex3}.}}
By Theorems \ref{ex1} and \ref{ex2},
$$
\ex_3^{(2)}(12;P)=40\quad \textrm{and }\quad\Ex^{(1)}_3(12;P) \cup \Ex^{(2)}_3(12;P) =\{S^3_{12}, K^3_6\cup K^3_6\}.
$$
 Therefore, to determine $\ex^{(3)}_3(12;P)$ we have to find the largest number of edges in a 12-vertex
  $P$-free 3-graph $H$   such that $H\nsubseteq S^3_{12}$ and $H \nsubseteq K^3_6\cup K^3_6$.
The comet $\co(12)$ satisfies all the above conditions and has 32 edges. Let $H$ be a 12-vertex
$P$-free 3-graph satisfying the above conditions but $H\neq \co(12)$.  Since $H\neq S^3_{12}$,
either $H$ forms a nontrivial intersecting family and, by Theorem \ref{nti},
$$
|H|\le  3\times 12-8 = 28<32,
$$
or $H\supset M$. We may thus consider  the latter case only. If $H$ is disconnected, then, since $H\nsubseteq K^3_6\cup K^3_6$,  by Theorems \ref{ex1} and \ref{PM},
\begin{eqnarray*}
    |H|\le \max\{\ex_3(7;P)+\ex_3(5;P),\ex_3(8;P)+\ex_3(4;P),\\
    \ex_3(9;P)+\ex_3(3;P),\ex_3(10;P|M),\ex_3(11;P|M)\}=\\
    \max\{20+10,21+4,28+1,24,30\}= 30<32.
\end{eqnarray*}
Assume, finally, that $H$ is connected and $H\supseteq M$.  If, in addition, $H\supseteq C$, then,
by Lemma \ref{spojny}, we have
$$
|H|\le  3\times 12-8=28 < 32.
$$
Otherwise, $H$ is a $\{P,C\}$-free 3-graph containing $M$. Therefore, by Theorem \ref{PCM},
$$
|H|< \ex_3(12;\{P,C\}|M)=4 + \binom{12-4}{2}=32,
$$
as  the comet $\co(12)$ is the only $M$-extremal 3-graph for $\{P,C\}$.
\qed

\bigskip

\noindent{\bf\textit{Proof of Theorem \ref{PM}.}} Recall, that we want to determine the conditional
Tur\'an number $\ex_3(n;P|M)$. By considering  whether or not a 3-graph contains a triangle, we
infer that
$$
\ex_3(n;P|M)=\max\{\ex_3(n;P|\{M,C\}),\ex_3(n;\{P,C\}|M)\}.
$$
The number $\ex_3(n;\{P,C\}|M)$ is given by Theorem \ref{PCM}, whereas
$$
\ex_3(n;P|\{M,C\})=\ex_3(n;P|C),
$$
since the unique extremal graph from Corollary \ref{coro} contains $M$.
One can easily check that for $6\le n \le 12$,
$$
\ex_3(n;P|\{M,C\})>\ex_3(n;\{P,C\}|M),
$$
 for $n=13$,
 $$
 \ex_3(n;P|\{M,C\})=\ex_3(n;\{P,C\}|M)=4+\binom{13-4}{2}=40,
    $$
     while  for $n \ge 14$,
     $$
     \ex_3(n;P|\{M,C\})<\ex_3(n;\{P,C\}|M).
        $$
Theorem \ref{PM} follows now immediately from the respective parts of Corollary \ref{coro} and Theorem~\ref{PCM}.
\qed

\section{Proofs of Lemmas \ref{spojny} and \ref{pcppm}}\label{proofspojny}

For a 3-graph $F$ and a vertex $v\in V(F)$ set $F(v)=\{e\in F: v\in e\}$. \emph{The degree of $v$
in $F$} is defined as $|F(v)|$.

\bigskip

\noindent{\bf\textit{Proof of Lemma \ref{spojny}.}}

Let $H$ be a $P$-free, connected 3-graph with $V(H)=V$ and $|V|=n\ge 7$, containing a triangle.
With some abuse of notation, we denote by $C$ a fixed copy of the triangle in $H$. Set
$$U=V(C)=\{x_1,x_2,x_3,y_1,y_2,y_3\},$$
and let, recalling that we identify the edge set of a 3-graph with the 3-graph itself,
$$ C=\{\{x_i,y_j,x_k\}:\; \{i,j,k\}=\{1,2,3\}\}.$$
Thus,  the vertices $x_1,x_2,x_3$ are of degree two in $C$, while $y_1,y_2,y_3$ are of degree one.
Further, let
$$W=V\setminus U,\qquad |W|=n-6$$
and let $H(U,W)$ denote the set of all edges of $H$ which intersect both $U$ and $W$.

It was observed in \cite{JPR}, Fact 1, that
$$H(U,W)=H\cap T,$$
where
$$T=T_1\cup T_2$$
and
$$T_1=\left\{\{x_i,y_i,w_l\}:\; 1\le i\le 3,\; 1\le l\le n-6\right\},$$
$$T_2=\left\{\{x_i,x_j,w_l\}:\;1\le i<j\le 3,\; 1\le l\le n-6\right\}.$$
Moreover, no edge of $H(U,W)$ may intersect an edge of $H[W]$, since otherwise there would be a
copy of $P$ in $H$ (\cite{JPR}, Fact 2). This and the connectivity assumption  imply that
$H[W]=\emptyset$. Thus,
$$H=H[U] \cup H(U,W),$$
and, clearly $H(U,W)\neq\emptyset$, as $W\neq\emptyset$ (see Fig. \ref{FigR3}).

\bigskip
\begin{figure}[!ht]
\centering
\includegraphics [width=7cm]{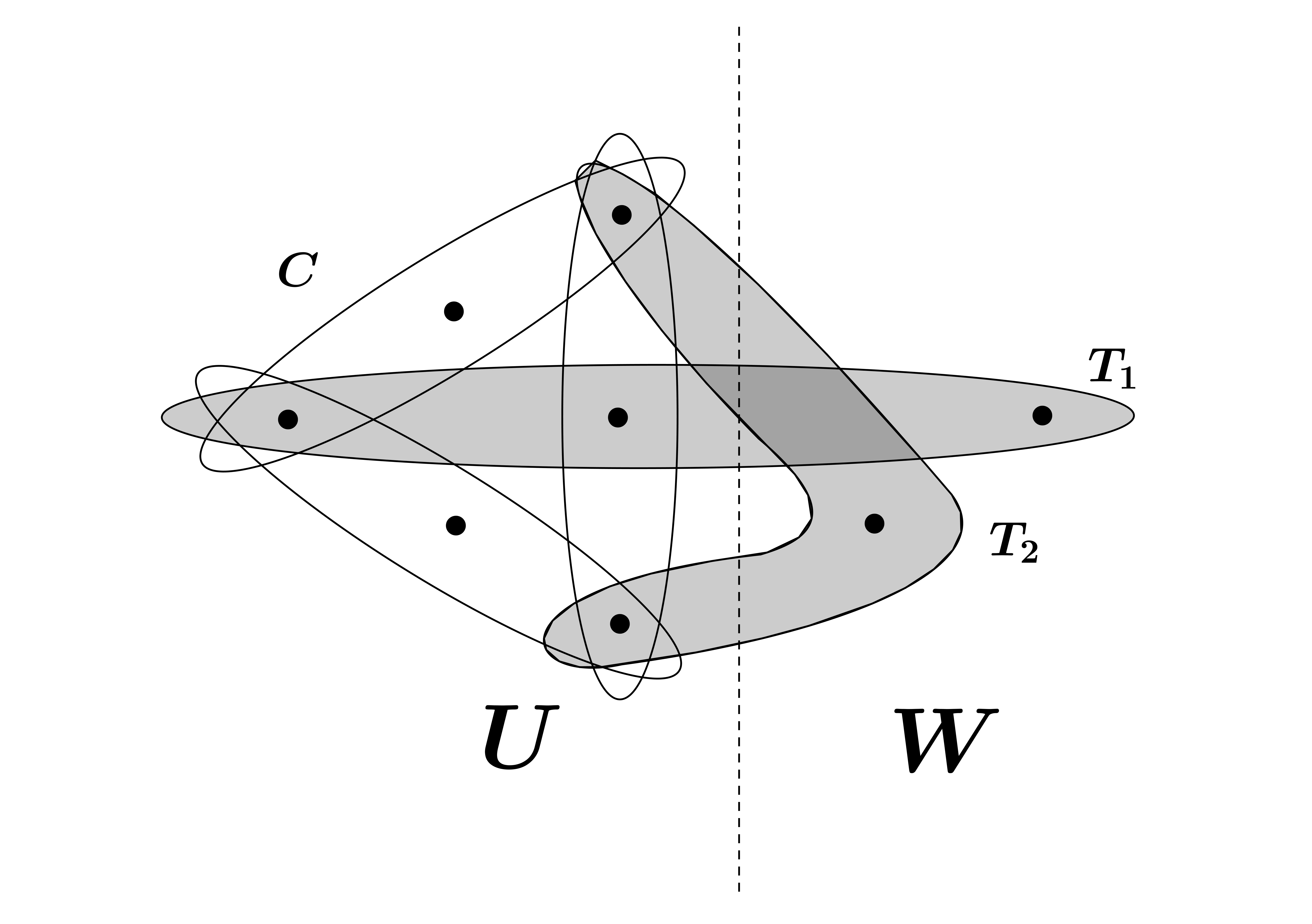}
\caption{Set-up for the proof of Lemma \ref{spojny}} \label{FigR3}
\end{figure}

If $H$ is an intersecting family (non-trivial due to the presence of $C$), then,  by
Theorem~\ref{nti}, $|H|\le 3n - 8$. We will show that if, on the other hand, $H\supseteq M$, then,
in fact, $|H|$ is even smaller. We begin with a simple observation.

\begin{fact}\label{inter}
    $H(U,W)$ is an intersecting family.
\end{fact}

\proof Recall that $H(U,W)\subseteq T_1\cup T_2$ and note that $T_2$ is intersecting by definition. On the
other hand, if $e\in T_1$, $f\in T$, and $e\cap f=\emptyset$, then $C\cup\{e\}\cup\{f\}\supset P$, so either $e$ or $f$ cannot be in $H$.
\qed

\bigskip

Let $f,h\in H$ satisfy $f\cap h = \emptyset$. By Fact \ref{inter}, at least one of $f$ and $h$
belongs to $H[U]$. If both of them were in $H[U]$ then, clearly, $f\cup h=U$ and,  by the
$P$-freeness of $H$, each  $e\in H(U,W)$  would need to be disjoint from one of them. In summary,
if $H\supseteq M$, then there exist two disjoint edges $e,f\in H$ such that $e \in H(U,W)$ and
$f\in H[U]$.

If $e\in T_1$, then one can easily check by inspection that $C\cup\{e\}\cup\{f\}\supset P$. Thus,
$e\in T_2$, say $e\cap U=\{x_1,x_2\}$. The only edge in $H[U]$ disjoint from $e$ which does not
create a copy of the path $P$ with $C\cup \{e\}$ is $f=\{x_3,y_1,y_2\}$  (see Fig. \ref{FigR4}).
Further, observe that all triples in $T$, except those of the type $\{x_1,x_2,w\}$, $w\in W$, form
a copy of $P$ with $f$ and some edge of $C$.

\bigskip
\begin{figure}[!ht]
\centering
\includegraphics [width=7cm]{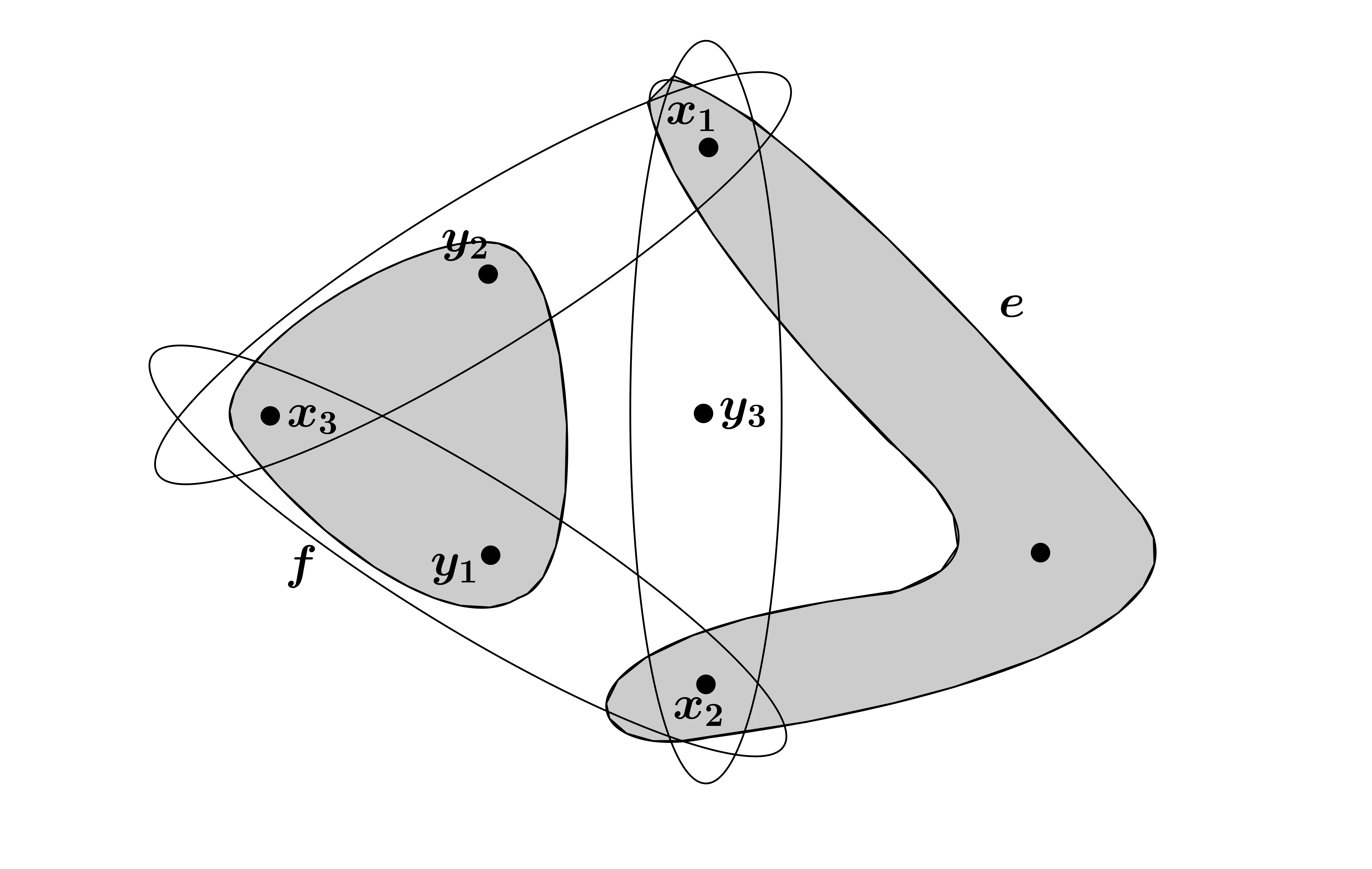}
\caption{Illustration to the proof of Lemma \ref{spojny}} \label{FigR4}
\end{figure}

Hence,
$$
H(U,W) \subseteq \{\{x_1,x_2,w\}: w\in W\}.
$$
and, consequently, $|H(U,W)|\le |W| = n-6$. Let
$$
X=\{\{y_1,y_2,y_3\},\{x_j,y_i,y_3\}, \{x_i,x_3,y_3\}, i\in \{1,2\}, j\in \{1,2,3\}\},
$$
Notice that $|X|=9$ and, for each $h\in X$, we have $C\cup \{e,f,h\}\supset P$. Thus, $H[U]\subseteq
\binom{U}{3}\setminus X$, so that $ |H[U]|\le 20-9=11. $ Consequently, for $n\ge7$,
$$
|H|=|H[U]|+|H(W,U)|\le 11 + n-6< 3n-8. \qed
$$

\bigskip

\noindent{\bf\textit{Proof of Lemma \ref{pcppm}.}} Let $V$ be a set with $|V|=n\ge 6$.
 Fix four vertices $v_1,v_2,v_3,v_4 \in V$ and define a 3-graph $H^{(0)}_n$ on $V$ as
    $$
    H^{(0)}_n=\left\{ h\in \binom{V}{3}:\quad  \{v_i,v_{i+1}\}\subset h,\quad i \in \{1,3\}\right\},
    $$
Note that  $H^{(0)}_n\supset M$ and $|H^{(0)}_n|=2n-4$.  Moreover, since every edge contains one of
the pairs $\{v_1,v_2\}$ or $\{v_3,v_4\}$,
      among any three edges  at least two share two vertices. Therefore, $H^{(0)}_n$ is
 $\{P,C,P^3_2\cup K_3^3\}$-free and, thus,
    $$\ex_3(n;\{P,C,P^3_2\cup K_3^3\}|M)\ge 2n-4.$$

    \bigskip

        To show the opposite inequality, consider a $\{P,C,P^3_2\cup K_3^3\}$-free
     3-graph $H$ containing $M=\{e,f\}$, with $V(H)=V$, $|V|=n\ge6$. Since $H$ is $P_2^3\cup
     K_3^3$-free, $H[V\setminus e]$ is $P^3_2$-free, and
     by Fact \ref{p2},
        $$
        |H[V\setminus e]|\le n-3\quad\mbox{ and }\quad|H[V\setminus f]|\le n-3.
        $$
      Also, since $H$ is $P$-free, there is no edge $h\in H$ with $|h\cap e|=|h\cap f| = 1$. Hence, if $|H[e\cup f]|=2$, then $|H| \le 2(n-3) = 2n-6$.

On the other hand, if there exists an edge $h\in H[e \cup f]\setminus\{e,f\}$, then, since $H$ is
$P^3_2\cup K_3^3$-free, all  edges of $H$ intersect one of $e$ or $f$ on at least two vertices. Let
$$
F_e = \{h\in H: |h\cap e|=2\}, \quad F_f =  \{h\in H: |h\cap f|=2\}.
$$
If there existed $h_1,h_2 \in F_e$ with $|h_1\cap h_2|=1$, then, depending on whether \newline
$|(h_1\cup h_2)\cap f|=0,1$, or 2, the edges $\{h_1,h_2,f\}$ would form, respectively, a
copy of $P^3_2\cup K_3^3$, $P$, or $C$ (see Fig. \ref{FigR5}).

\bigskip
\begin{figure}[!ht]
\centering
\includegraphics [width=10cm]{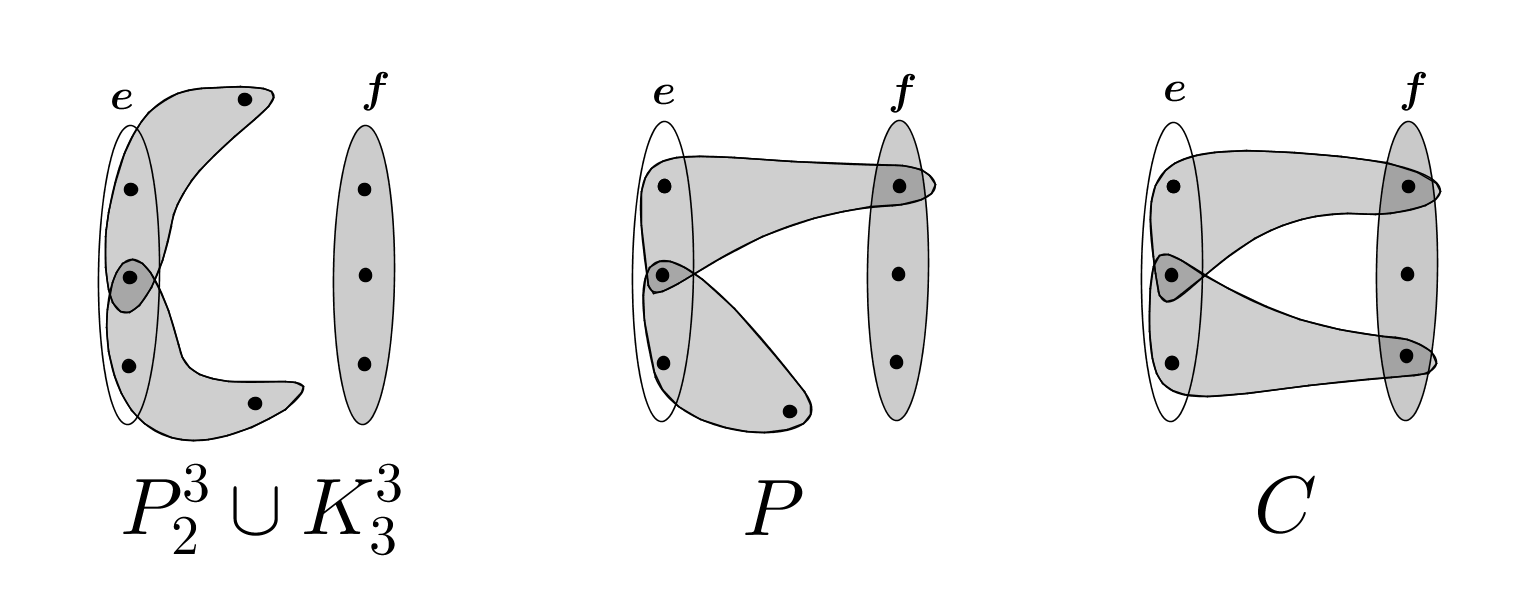}
\caption{Illustration to the proof of Lemma \ref{pcppm}} \label{FigR5}
\end{figure}

Thus,
$$
\forall h_1,h_2 \in F_e, \quad |h_1\cap h_2|=2,
$$
so, either all pairs $h_1,h_2\in F_e$ share two vertices of $e$ or all pairs $h_1,h_2\in F_e$ share one vertex of $V\setminus e$ (and another in $e$)

This implies that
$$
|F_e| \le \max\{n-3,3\}=n-3.
$$
Similarly, $|F_f|\le n-3$ and,
consequently,
$$
|H| = |\{e,f\}| + |F_e|+|F_f| \le 2 + (n-3) + (n-3)=2n-4. \qed
$$

\section{Proof of Theorem \ref{PCM}}

This section is entirely devoted to proving Theorem \ref{PCM}, that is, to determining the largest number of edges in an $n$-vertex 3-graph which is $P$-free and $C$-free but is not an intersecting family.

First note that since $|V(P^3_2\cup K_3^3)|=8$, no $n$-vertex 3-graph, $n=6,7$, contains a copy of $P^3_2\cup K_3^3$ and therefore, by Lemma \ref{pcppm},
$$
\ex_3(n;\{P,C\}|M)=\ex_3(n;\{P,C,P^3_2\cup K_3^3\}|M)=2n-4.
$$
 Thus, from now on we will be assuming that $n\ge8$.
 Define a sequence of 3-graphs
 $$
 H_n=\left\{ \begin{array}{ll}
 H_n^{(0)} & \textrm{ for } 8\le n\le 9,\\
 K_5^3\cup K_5^3 & \textrm{ for } n=10,\\
 \co(n) & \textrm{ for } n\ge 11,
 \end{array} \right.
 $$
 where   $H_n^{(0)}$ is the 3-graph introduced in the proof of Lemma \ref{pcppm}.
 By  simple inspection one can see that $H_n$ is $\{P,C\}$-free and contains $M$. Hence
    $$
    \mathrm{ex}_3(n;\{P,C\}|M)\ge |H_n|=
    \left\{\begin{array}{ll}
    2n-4 &\textrm{for } 8\le n\le 9,\\
    20 & \textrm{for }n=10,\\
    4 + \binom{n-4}{2} &\textrm{for } n\ge 11.
    \end{array}\right.
    $$
The main difficulty lies in showing the reverse inequality, namely, that any $\{P, C\}$-free 3-graph $H$ on  $n\ge8$ vertices, containing  $M$,  satisfies $|H|\le |H_n|$. Moreover, for $n\ge11$, we want to show that the equality is reached by the extremal 3-graph $H_n=\co(n)$ only.
We may assume that $H$ contains a copy of $P^3_2\cup K_3^3$, since otherwise, by Lemma \ref{pcppm},
$$|H|\le 2n-4\le |H_n|,$$
where the last inequality is strict for $n\ge10$. Before we turn to the actual proof of Theorem \ref{PCM}, we need to introduce some notation and prove preliminary results about the structure of $H$.

\subsection{Preparations for the proof}\label{prepa}

We assume that $H$ is $\{P, C\}$-free and contains a copy of $P^3_2\cup K_3^3$.
 Let $e_1,e_2\in H$ and $x\in V=V(H)$ be such that $e_1\cap e_2=\{x\}$ and there is an edge in $H$ disjoint from $e_1\cup e_2$. We know that such a choice of $e_1,e_2,x$ exists, because  $H\supseteq P^3_2\cup K_3^3$. We split  $V=U\cup W$, where
 $$
 U=e_1\cup e_2,\quad \text{and}\quad  W=V\setminus U.
 $$
  Note that  $|U|=5$ and  $|W|=n-5$.  Further set
  $$
   H(U,W)=H\setminus (H[U]\cup H[W])
  $$
  for the sub-3-graph of $H$ consisting of all edges intersecting both, $U$ and $W$.
   Notice that  $H[W]\neq\emptyset$, and thus the set $W_0$  of vertices of degree 0 in $H[W]$ has size
   \begin{equation}\label{n-8}
   |W_0|\le n-8.
   \end{equation} Set also $W_1=W\setminus W_0$
(see Fig. \ref{FigR6}).

\bigskip
\begin{figure}[!ht]
\centering
\includegraphics [width=9cm]{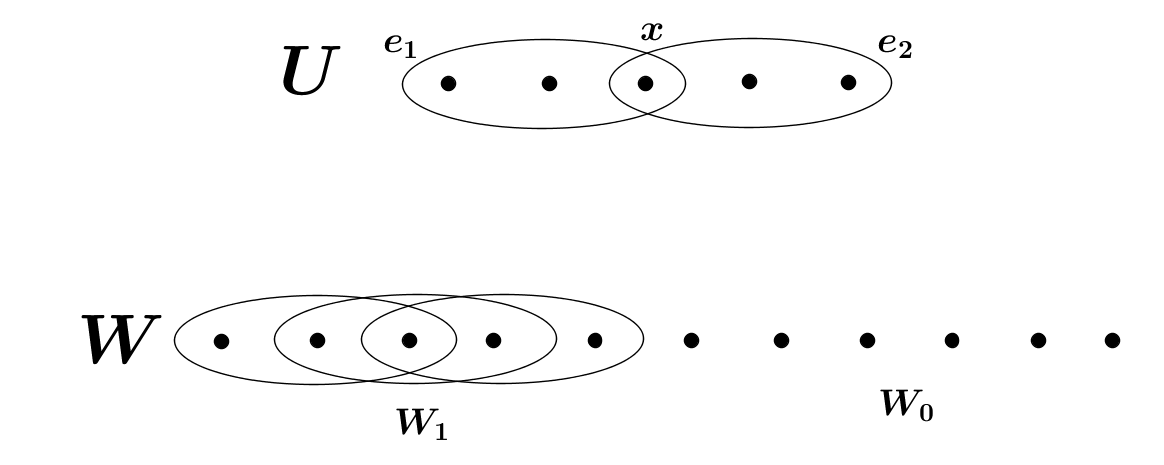}
\caption{Set-up for the proof of Theorem \ref{PCM}} \label{FigR6}
\end{figure}
Let us split
$$H[U]=\{e_1,e_2\}\cup E(x)\cup E(\bar x),$$
where $E(x)$ contains all edges of $H[U]$ which contain vertex $x$, except for $e_1$ and $e_2$, while $E(\bar x)$ contains all other edges of $H[U]$.
Note that
\begin{equation}\label{eq2}
\max\{|E(x)|, |E(\bar x)|\}\le 4.
\end{equation}

\bigskip

We also split the set of edges of $H(U,W)$.
First, notice that if for some $h\in H(U,W)$ we have $|h\cap U|=1$, then $h\cap U = \{x\}$, since otherwise $h$ together with $e_1$ and $e_2$ would form a copy of $P$ in $H$. We let
$$
F^0=\{h\in H(U,W): h\cap U=\{x\}\}.
$$

The edges $h\in H(U,W)$ with $|h\cap U|=2$ must satisfy $h\cap U \subset e_1$ or $h\cap U\subset e_2$,
since otherwise $h$ together with  $e_1$ and $e_2$ would form a copy of $C$ in $H$. For $k=1,2$ define
$$
F^k=\{h\in H(U,W): \quad |h\cap U\setminus\{x\}|=k\}.
$$
We have $H(U,W)=F^0\cup F^1\cup F^2$. (Note that in each case $k=0,1,2$, the superscript $k$ stands for the
common size of the set $h\cap U\setminus\{x\}$ -- see Fig.
\ref{FigR7}.)

\bigskip
\begin{figure}[!ht]
\centering
\includegraphics [width=7cm]{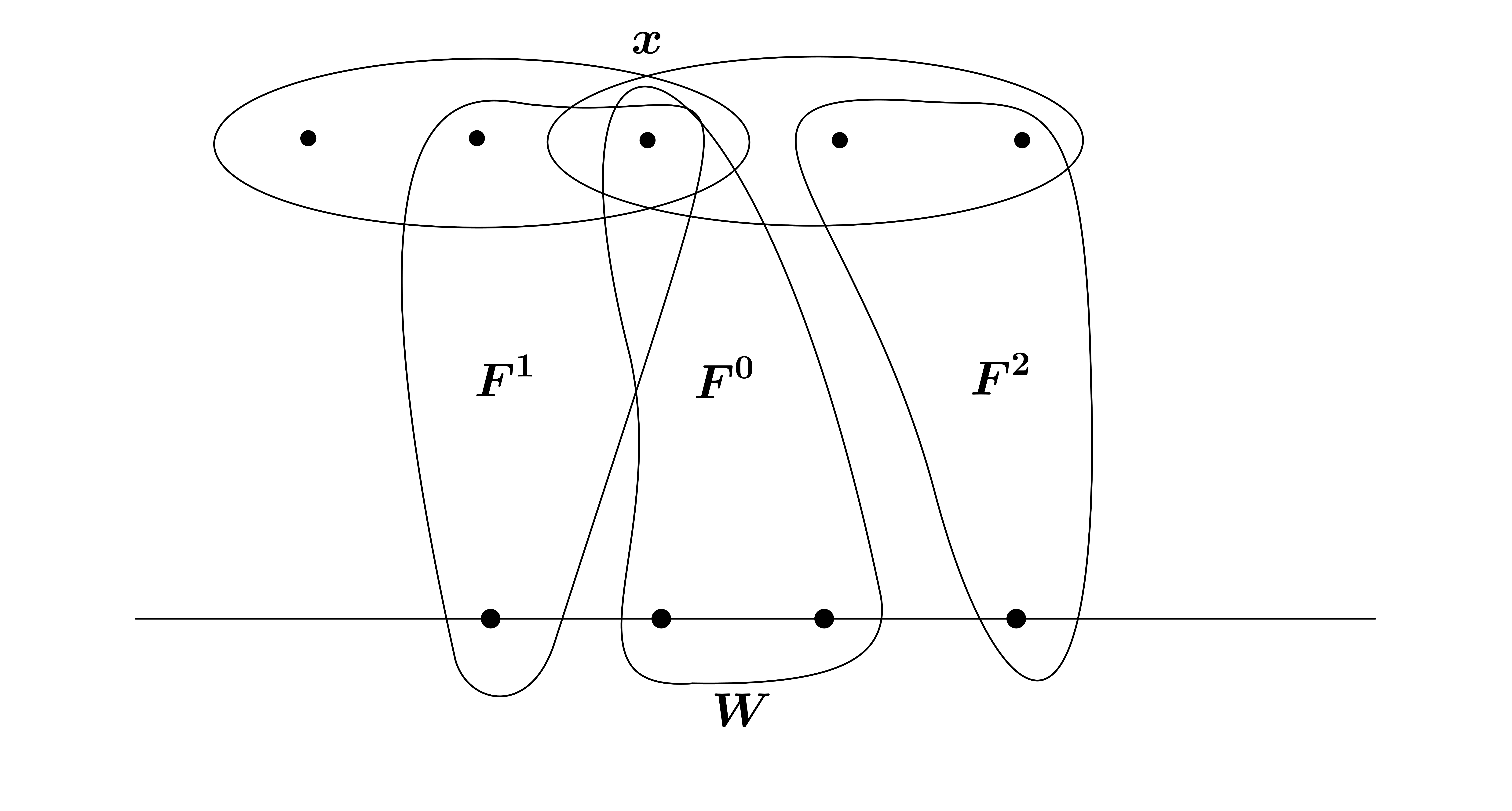}
\caption{Three types of edges in $H(U,W)$} \label{FigR7}
\end{figure}

 For a sub-3-graph $F\subseteq H(U,W)$ and  $i=0,1$, set
$$
F_i=\{h\in F: h \cap W \subset W_i\},
$$
which in the  important case of $F=H(U,W)$ will be abbreviated to $H_i$. In particular, for
$i=0,1$, $H_i=F_i^0\cup F_i^1\cup F_i^2$, where $F_i^0$ is the subset of edges  $h \in F^0$ with
$|h\cap W_i|=2$, while $F_i^k$, $k=1,2$, is the subset of edges  of $F^k$ whose unique vertex in
$W$ lies in $W_i$.

    A simple but crucial observation is that, since $H$ is $P$-free, for every two  disjoint
 edges in $H$,   no  edge  may intersect each of them in exactly one vertex.
 % We will be referring to it as \emph{the no-one-one-principle}.
    Thus, there is no edge in $H$ with one vertex in each of the sets, $U$, $W_0$ and $W_1$. Therefore,
\begin{equation}\label{HHH}
    H(U,W)=H_0\cup H_1,
\end{equation}
    and consequently,
    \begin{equation}\label{r2}
    \begin{aligned}
    H &= H[U]\cup H(U,W)\cup H[W]=H[U]\cup H_0\cup H_1\cup H[W]\\
    &= H[U\cup W_0]\cup H_1\cup H[W].
    \end{aligned}
    \end{equation}
Furthermore, by the same principle, if $e\in F^0_1$, then the pair $e\cap W_1$ must be \emph{nonseparable} in $H[W_1]$, that is, every edge of $H[W_1]$ must contain both these vertices or none.  Since, as it can be easily proved, there are at most $|W_1|$ nonseparable pairs in $W_1$,
    \begin{equation}\label{r9}
     |F^0_1|\le |W_1|.
    \end{equation}

%To bound from above the number of edges in $H$ we will observe that some edges, if exists together, form a path $P$ or a triangle $C$ and therefore exclude each %other.

Another consequence of the above observation  is that $F^1_1=\emptyset$. Thus,
\begin{equation}\label{r6}
H_1=F_1^0\cup F_1^2.
\end{equation}
To make use of (\ref{r6}), in addition to (\ref{r9}), we  need to bound  $|F_1^2|$ which, however, requires a detailed analysis of
the degrees of vertices $v\in W$ in the 3-graphs $F^k$, $k=0,1,2$.
For  $v\in W$ and $F\subseteq H$, denote by $F(v)$ the degree of $v$ in $F$.

It can be easily checked that, since $H$ is $P$-free, for every $v\in W$ either
 \begin{equation}\label{FORF}
 F^0(v)=\emptyset\quad\mbox{ or }\quad F^2(v)=\emptyset.
  \end{equation}
  Moreover, by the definitions of $F^1$ and $F^2$,
  \begin{equation}\label{4and2}
  |F^1(v)|\le 4\quad{and }\quad|F^2(v)|\le 2.
  \end{equation}

  For $v\in W_0$, by the remark preceding (\ref{HHH}), $|F^0(v)|\le |W_0|-1$, and thus, by (\ref{FORF}), (\ref{4and2}), and (\ref{n-8}),
$$
 |H(v)|=|F^0(v)|+ |F^1(v)|+|F^2(v)|\le 4+  \max\{2,n-9\}.
$$
In particular,  for $n=10$,
\begin{equation}\label{o16}
\forall v \in W_0,\quad |H(v)| \le 6,
\end{equation}
while for $n\ge 11$,
\begin{equation}\label{o17}
\forall v \in W_0,\quad |H(v)| \le n-5,
\end{equation}
where the equality for $n\ge 12$ is achieved only when $|F^0(v)|=n-9$, $|F^1(v)|=4$, and $F^2(v)=\emptyset$.

 Consider now $v\in W_1$. For each $e\in F^0$, the pair $e\cap W$ must be nonseparable and $v$ belongs to at most two nonseparable pairs. Thus, $|F^0(v)|\le 2$ and, consequently, by (\ref{FORF}) and (\ref{4and2}),
\begin{equation}\label{r8}
\forall v \in W_1,\quad|H_1(v)|=|F^0(v)| + |F^2(v)|\le 2.
\end{equation}

    One can also show, that
        \begin{equation}\label{r10}
        |F_1^2|\le \max\{|W_1|,2|W_1|-4\}.
        \end{equation}
        Indeed, if for all $v\in W_1$ we have $|F^2(v)|=|F^2_1(v)|=1$, then $|F^2_1|\le |W_1|$. Otherwise, let  $v\in W_1$ have, by (\ref{4and2}), $|F^2(v)|=2$ and let $\{v,v',v''\}\in H[W]$. Since $H$ is $P$-free, $F^2(v')=F^2(v'')=\emptyset$, and therefore, again by (\ref{4and2}),
        $$
        |F^2_1|\le 2(|W_1|-2)=2|W_1|-4.
        $$

Now we are ready to set bounds on the number of edges in $H_1$, as well as in $H[U]\cup H_1$, which will be repeatedly used in the proof of Theorem \ref{PCM}. Recall that $|W_1|\ge3$.
\bigskip

\begin{fact} We have
\begin{equation}\label{r4}
|H_1|\le 2|W_1|-3
\end{equation}
and, for $|W_1|\ge 4$,
\begin{equation}\label{r0}
|H[U]| + |H_1|\le 2|W_1|+2.
\end{equation}
\end{fact}

\proof

Let $h\in H[W]$. It is easy to check by inspection that $\sum_{v\in h}H_1(v)\le3$, while for $v\in
W_1\setminus h$, by~(\ref{r8}),  $|H_1(v)|\le 2$. This yields $|H_1|\le3 + 2(|W_1|-3)$ and takes
care of~(\ref{r4}).

If $H_1=\emptyset$ then (\ref{r0}) holds, as $|H[U]|\le10$.
To prove (\ref{r0}) also when $H_1\neq\emptyset$, we need a better bound on $|H[U]|$. To this end, note that if $F^0\neq\emptyset$ then $E(\bar x)=\emptyset$, while if  $F^2_1\neq\emptyset$ then $E(x)=\emptyset$.
Hence, by (\ref{r6}) and~(\ref{eq2}),
\begin{equation}\label{o12}
H_1\neq\emptyset\qquad\Rightarrow\qquad|H[U]|\le 6.
\end{equation}
So, if one of the sets, $F^0_1$ or $F^2_1$, is empty, then we get (\ref{r0}) by (\ref{o12}), (\ref{r9}), and (\ref{r10}). If both these sets are nonempty, then   $E(\bar x)=E(x)=\emptyset$, and thus $|H[U]|=2$. In this case (\ref{r0}) follows  by (\ref{r4}) with a margin.
\qed

\medskip

Since $H$ is $C$-free, on several occasions our proof  relies on two instances of Theorem~\ref{c3}. Namely, if $|W_0|\ge 1$ then
    \begin{equation}\label{r3}
    |H[U\cup W_0]|\le \binom{|W_0|+4}{2},
    \end{equation}
while if $|W|=n-5\ge 6$ then
\begin{equation}\label{hw}
|H[W]| \le \binom{n-6}{2}.
\end{equation}

Finally, there cannot be too many edges between $U$ and the vertex set of a copy of $P_2^3$ in $H[W]$ if there happens to be one.
For a subset $W'\subset W$, we denote by $H(U,W')$ the sub-3-graph of $H$ consisting of all edges intersecting both, $U$ and $W'$.

\medskip

\begin{fact}\label{o5}
    If $H[W]$ contains a copy $Q$ of $P^3_2$ with $V(Q)\subseteq W$, then
    \begin{equation}\label{eq7}
    |H(U,V(Q)) |\le 4.
    \end{equation}
   % Moreover, there is no  edge $h=\{x_1,x_2,x_3\}\in H$ such that $x_1 \in U$, $x_2 \in V_Q$ and $x_3\in W\setminus V_Q$.
\end{fact}
\begin{proof}
    Note that, due to $P$-freeness of $H$, the only edges allowed in $H(U,V(Q))$ with one vertex in $U$ must
 belong to $F_0$ (there are at most two such edges). By symmetry, there are also at most two edges in
 $H(U,V(Q))$ with one vertex in $W$, which yields~(\ref{eq7}).
     %The second assertion follows by the no-one-one-principle.
\end{proof}

\subsection{The proof}

The structure of the proof is as follows. We first settle the three smallest cases, $n=8,9,10$, one by one. Then we turn to the main case of $n\ge11$. Here, after quickly taking care of the easy subcase  $W_0=\emptyset$, we assume that $W_0\neq\emptyset$ and proceed by induction on $n$ with $n=11$ being the base case. This part is a bit pedestrian, but afterwards, the induction step is almost immediate.

Let $H$ be a $\{P, C\}$-free $n$-vertex 3-graph which contains a copy of $P^3_2\cup K_3^3$. We adopt the notation and terminology from Subsection \ref{prepa}.
In addition, for $v\in V$, we will write $H-v$ for  $H[V\setminus\{v\}]$.

\bigskip

\noindent{\bf$\mathbf{n=8}$.}
We have $|W_1|=3$, $|H[W]|=1$, and $W_0=\emptyset$.
If $H(U,W)=H_1=\emptyset$, then
$$
|H|=|H[U]|+|H[W]|\le 10 + 1<12=|H_8|,
$$
Otherwise, by (\ref{o12}), $|H[U]|\le 6$ and, therefore, by (\ref{r4}),
$$
|H|=|H[U]| + |H_1| + |H[W]| \le 6 + 3 + 1<12.
$$

\bigskip

\noindent{\bf$\mathbf{n=9}$.}
We have $|W|= 4$ and $|H[W]|\le \binom{4}{3}=4$.
If $W_0=\emptyset$ then, by (\ref{r0}),
$$
|H|=|H[U]| + |H_1| + |H[W]| \le 2|W|+2 + 4 = 14=|H_9|.
$$

If $W_0\neq\emptyset$ then $|W_0|=1$, $|W_1|=3$ and $|H[W]|=1$.
By (\ref{r4}), $|H_1|\le 3$, and
consequently, by (\ref{r2}) and (\ref{r3}),
$$
|H|=|H[U\cup W_0]|+|H_1| + |H[W]|\le 10 + 3 + 1=14.
$$

\bigskip

\noindent{\bf$\mathbf{n=10}$.}
We have $|W|=5$, $|W_0|\le 2$ and $|H[W]|\le \binom{5}{3}=10$.
If $W_0=\emptyset$ then, by (\ref{r0}), $|H[U]|+|H_1 |\le 2|W|+2=12$.
If, additionally, $|H[W]|\le 5$, then
$$
|H|=|H[U]| + |H_1| + |H[W]| \le 12 + 5< 20 =|H_{10}|.
$$
Otherwise, by Fact \ref{p2}, $H[W]$ contains a copy $Q$ of $P_2^3$ (note that $V(Q)=W_1$), and, by~(\ref{eq7}), $|H_1|\le 4$. Hence, using (\ref{o12}) along the way,
$$
|H|=|H[U]| + |H_1| + |H[W]| \le\max\{10+0, 6 + 4\}+ 10= 20.
$$

Now, let $W_0\neq\emptyset$. Fix $v\in W_0$ and notice that $H-v$ is $\{P,C\}$-free and contains $M$. Since  we have already proved that $\ex_3(9;\{P,C\}|M)=14$,
$$
|H-v|\le 14.
$$
Moreover, by (\ref{o16}), $|H(v)|\le 6$, and consequently,
$$
|H|=|H-v|+|H(v)|\le 14+6=20.
$$

\bigskip

\noindent {\bf $\mathbf{n\ge11.}$} The proof is by induction on $n$ with $n=11$ being the base case. First, however, we take care of a simple subcase when $W_0=\emptyset$, for which, by (\ref{r0}) and (\ref{hw}),
$$
|H|=|H[U]| + |H_1| + |H[W]| \le 2(n-5) + 2 + \binom{n-6}{2}= 3+\binom{n-4}{2}< |H_n|.
$$
Hence, in what follows we will be assuming that $W_0\neq\emptyset$.

\bigskip

\noindent {\bf $\mathbf{n=11}$ (base case).} Suppose first that $H[W]$ contains a copy $Q$ of $P^3_2$. Then $|W_0|=1$, $|W_1|=5$, $V(Q)=W_1$, and by (\ref{eq7}), $|H_1|\le 4$.
    Consequently, by (\ref{r2}), (\ref{r3}), and  (\ref{hw}),
    $$
    |H|=|H[U\cup W_0]|+|H_1| + |H[W]|\le 10 + 4 + 10< 25= |H_{11}|.
    $$

In the remainder of this part of the proof, besides the assumption that $W_0\neq\emptyset$, we will be also assuming that $H[W]$ is $P^3_2$-free and thus, by Fact \ref{p2}, $|H[W]|\le 6$.
We consider three cases with respect to the size of $|W_0|$.

    \bigskip

    \noindent {$\mathbf{|W_0|=1.}$}  We have $|W_1|=5$ and, by (\ref{r4}), $|H_1|\le 7$. Consequently, by (\ref{r2}) and (\ref{r3}),
        $$
            |H| =|H[U\cup W_0]|+|H_1|+|H[W]|\le 10+7+6  < 25.
        $$

    \noindent {$\mathbf{|W_0|=2.}$} We have $|W_1|=4$ and therefore $|H[W]|\le \binom{4}{3}=4$. Moreover, by (\ref{r4}), $|H_1|\le 5$ and finally, by (\ref{r2}) and (\ref{r3}),
        $$
        |H| =|H[U\cup W_0]|+|H_1|+|H[W]|\le 15+5+4  < 25.
        $$

  \noindent { $\mathbf{|W_0|=3.}$} We have $|W_1|=3$ and therefore $|H[W]|=1$. Moreover, by (\ref{r4}), $|H_1|\le 3$ and thus, by (\ref{r2}) and (\ref{r3}),
  $$
    |H| =|H[U\cup W_0]|+|H_1|+|H[W]|\le 21+3+1 = 25,
  $$
  with equality only when $|H_1|=3$ and $|H[U\cup W_0]|=21$.
    The latter, by the second part of Theorem \ref{c3}, is possible only when $H[U\cup W_0]$ is a star (with the center at $x$). This, in turn, implies that $F^2=\emptyset$ (otherwise $H$ would not be $P$-free) and, further, by (\ref{r6}), that $H_1=F^0_1$. Hence, $H=\co(11)$ with $x$ at the center and $W_1$ as the head.

\bigskip

\noindent {\bf $\mathbf{n\ge 12}$ (inductive step).}
Fix $v\in W_0$. By the induction hypothesis
$$
|H-v|\le 4+\binom{n-5}{2}
$$
with the equality only when $H-v=\co(n-1)$. Looking at the structure of $H-v$, if it is a comet, then it must have the center at $x$ and the head must be the unique edge of $H[W]$.
Moreover, by (\ref{o17}),
$|H(v)|\le n-5$,
with the equality only when $|F^0(v)|=n-9$, $|F^1(v)|=4$, and $F^2(v)=\emptyset$.
 Consequently,
$$
|H|=|H-v|+|H(v)|\le 4+\binom{n-5}{2}+n-5=|H_{n}|.
$$
and this bound is achieved only when both $H-v=\co(n-1)$ and $|H(v)|=n-5$. This, however, implies that $H=\co(n)$ (with the same center and head as in $H-v$.) Theorem \ref{PCM} is proved.

\section{Final comments}\label{final}

It would be interesting to decide if $R(P;r)=r+6$ for all $r$. If not, then what is the largest $r_0$ such that $R(P;r)=r+6$ for all $r\le r_0$?
To even partially answer these questions, we would need to compute the conditional Tur\'an numbers $\ex^{(s)}(n;P|M)$ for $s\ge3$.

For the related problem of computing $R(C;r)$ it is only known that $R(C;r)=r+5$ for $r=2,3$  and $R(C;r)\ge r+5$ for all $r$ (\cite{GR}).
Gyarfas and Raeisi conjecture in \cite{GR} that $R(C;r)=r+5$ for all $r$. To facilitate  our approach to this problem one would need to compute $\ex_3^{(s)}(n;C)$ for  $s\ge2$ and some small values of $n$. This would probably include calculating the conditional Tur\'an numbers $\ex_3(n;C|M)=\ex_3(n;C|P)$ which might be of independent interest. (The fact that the two numbers are the same was derived in \cite{JPR} from Theorem \ref{PM} which was conjectured there.)
In \cite{JPR} we  showed that $ex_3(n;C|M)\ge\binom{n-2}2+1$ and conjectured that, indeed, this lower bound is the true value of $ex_3(n;C|M)$.


\begin{thebibliography}{4}
    %\bibitem{BK}
    %N. Bushaw, N. Kettle \textit{Tur\'an Numbers for Forests of Paths in Hypergraphs.}
    \bibitem{CK}
    R. Cs\'{a}k\'{a}ny, J. Kahn, \textit{A homological Approach to Two Problems on Finite Sets}, Journal of Algebraic Combinatorics 9 (1999), 141-149.
    \bibitem{EKR}
    P. Erd\"os, C. Ko, R. Rado, \textit{Intersection theorems for systems of finite sets}, Quart. J. Math. Oxford Ser. (2) 12 (1961), 313-320.
    %\bibitem{FLPS}
    %R.J. Faudree, S.L. Lawrence, T.D. Parsons, R.H. Schelp, \textit{Path-Cycle Ramsey Numbers},
    %Discrete Math. 10 (1974), 269-277.
    %\bibitem{FS}
    %R.J. Faudree, R.H. Schelp, \textit{All Ramsey Numbers for Cycles in Graphs}, Discrete
    %Math. 8 (1974), 313-329.
    \bibitem{F}
    P. Frankl, \textit{On families of fnite sets no two of which intersect in a singleton}, Bull. Austral. Math. Soc. 17 (1977), 125-134.
    \bibitem{FF}
    P. Frankl, Z. F\"{u}redi, \textit{Exact solution of some Tur\'{a}n-type problems}, J. Combin. Th. Ser. A 45 (1987), 226-262.
    %\bibitem{FF2}
    %P. Frankl, Z. F\"{u}redi, \textit{Non-trivial Intersecting Families}, J. Combin. Th. Ser. A 41 (1986), 150-153.
    \bibitem{FJS}
    Z. F\"{u}redi, T. Jiang, R. Seiver, \textit{Exact solution of the hypergraph Tur\'{a}n problem for k-uniform linear paths}, Combinatorica 34 (3) (2014), 299-322.
    \bibitem{OnRamsey}
    L. Gerencs\'{e}r, A. Gy\'{a}rf\'{a}s, \textit{On Ramsey-Type Problems}, Annales Universitatis Scien-tiarum Budapestinensis, E\"{o}tv\"{o}s Sect. Math. 10 (1967), 167-170.
    \bibitem{GR}
    A. Gy\'{a}rf\'{a}s, G. Raeisi, \textit{The Ramsey number of loose triangles and quadrangles in hypergraphs}, Electron. J. Combin. 19 (2012), no. 2, \# R30.
    \bibitem{GSS}
    A. Gy\'{a}rf\'{a}s, G. S\'{a}rk\"{o}zy, E. Szemer\'{e}di, \textit{The Ramsey number of diamond-matchings and loose cycles in hypergraphs}, Electron. J. Combin. 15 (2008), no. 1, \# R126.
    %\bibitem{HR}
    %P. Haxell, T. {\L}uczak, Y. Peng, V. R\"{o}dl, A. Ruci\'{n}ski, M. Simonovits, J. Skokan, \textit{The Ramsey number for hypergraph cycles I}, J. Combin. Theory, Ser. A, 113 (2006), 67-83.
    \bibitem{HM}
    A.J.W. Hilton, E.C. Milner, \textit{Some intersection theorems for systems of finite sets}, Quart. J. Math. Oxford Ser. (2) 18 (1967), 369-384.
    \bibitem{J}
    E. Jackowska \textit{The 3-colored Ramsey number of 3-uniform loos paths of length 3}, submitted.
    \bibitem{JPR}
    E. Jackowska, J. Polcyn, A Ruci\'nski, \textit{Tur\'an numbers for linear 3-uniform paths of length 3}, submitted.
    %\bibitem{Keev}
    %P. Keevash, \textit{Hypergraph Tur\'{a}n problems}, Surveys in Combinatorics 2011, Cambridge University Press, 2011, 83-140.
    \bibitem{KMW}
    P. Keevash, D. Mubayi, R. M. Wilson, \textit{Set systems with no singleton intersection}, SIAM J. Discrete Math. 20 (2006), 1031-1041.
    %\bibitem{Karos}
    %G. K\'{a}rolyi, V. Rosta, \textit{Generalized and Geometric Ramsey Numbers for Cycles}, Theoretical Computer Science, 263 (2001) 87-98.
    \bibitem{Kostochka}
    A. Kostochka, D. Mubayi, J. Verstra\"{e}te, \textit{Tur\'{a}n Problems and Shadows I: Paths and Cycles},  J. Combin. Theory Ser. A 129 (2015), 57–79.
    %\bibitem{MaORS}
    %L. Maherani, G.R. Omidi, G. Raeisi, M. Shahsiah, \textit{The Ramsey Number of Loose Paths in 3-Uniform Hypergraphs}, Electronic Journal of Combinatorics, http://www.combinatorics.org, \# P12, 20(1)
    %(2013), 8 pages.
    %\bibitem{MaORS2}
    %L. Maherani, G.R. Omidi, G. Raeisi, M. Shahsiah, \textit{On Three-Color Ramsey Number of Paths}, preprint, arXiv, http://arxiv.org/abs/1207.3771 (2012).
    \bibitem{Omidi}
    G.R. Omidi, M. Shahsiah, \textit{Ramsey Numbers of 3-Uniform Loose Paths and Loose Cycles},
      \emph{J. Comb. Theory, Ser. A} 121 (2014) 64-73.
    %\bibitem{Radzisz}
    %S. P. Radziszowski, \textit{Small Ramsey numbers}, Electronic Journal of Combinatorics 1 (1994), Dynamic Surveys, DS1.13 (August 22, 2011).
    %\bibitem{Rosta}
    %V. Rosta, \textit{On a Ramsey Type Problem of J.A. Bondy and P. Erdo˝s, I and II}, Journal of Combinatorial
    %Theory, Series B, 15 (1973) 94-120.
\end{thebibliography}
\end{document}